\documentclass[11pt]{article}
\usepackage[utf8]{inputenc}

\usepackage{comment}
\usepackage{lineno}
\usepackage{etex}
\usepackage[utf8]{inputenc}
\usepackage{amsmath,amssymb,amsthm,mathtools}
\usepackage{epstopdf}
\usepackage{bbm}		
\usepackage{mathtools} 
\usepackage{enumerate}
\usepackage{a4wide}
\usepackage{graphics} 
\usepackage{graphicx} 
\usepackage{epstopdf}
\usepackage{a4wide}
\usepackage{nicefrac}

\newcommand{\R}{\mathbb{R}}

\newcommand{\diag}{\textup{diag}}

\newcommand{\M}{\mathcal{M}}

\def \T{\textup{T}}

\def \d{\,\textup{d}}
\def \G{\mathcal{G}}
\def \Veq{V_{\textup{eq}}}

\def \D{\textup{D}}

\def \indikator{\mathbbm{1}}
\def \zeros{\mathbb{O}}
\def \P{\mathcal{P}}

\def \u{\mathbf{u}}
\def \f{\mathbf{f}}
\def \S{\widehat{\mathbf{S}}}
\def \hesitation{h}
\def \hesitationhat{\hat{h}}
\def \hesitationJacobi{\widehat{h'}}

\def \DVeq{{\mathcal{D}_{\Veq'}}}
\def \Dh{{\mathcal{D}_{h'}}}

\def \veco{\vec{0}}
\def \uhat{\mathbf{\hat{u}}}
\def \fhat{\mathbf{\hat{f}}}

\def \velocityhat{{\widehat{w}}}
\def \velocity{{w}}
\def \vI{\widehat{v^{(1)}}}

\def \Deq{{\widehat{\lambda_{\textup{eq}}}}}

\def \NumFlux{\boldsymbol{\widehat{F}}} 
\def \caStoConservedLeft{\boldsymbol{\bar{u}_\ell}}
\def \caStoConservedRight{\boldsymbol{\bar{u}_r}}
\def \caStoConserved{\boldsymbol{\bar{u}}}

\newtheorem{theorem}{Theorem}[section]

\newtheorem{corollary}[theorem]{Corollary}
\theoremstyle{remark}

\usepackage{authblk}

\begin{document}

\title{Stability analysis of a hyperbolic stochastic Galerkin formulation for  the Aw-Rascle-Zhang model with relaxation}

\author[]{
	Stephan Gerster, 
	Michael Herty,
	Elisa Iacomini
}
\affil[]{ \em \small Institut für Geometrie und Praktische Mathematik (IGPM),\\ \em RWTH Aachen University, \\ \em Templergraben 55, 52062 Aachen, Germany}

\date{}                     

\maketitle

\begin{abstract}

We investigate the propagation of uncertainties in the Aw-Rascle-Zhang model, which belongs to a class of second order traffic flow models described by a system of nonlinear hyperbolic equations.  The stochastic quantities   
are expanded in terms of wavelet-based series expansions. 
Then, they are projected to obtain a deterministic system for the coefficients in the truncated series. 
Stochastic Galerkin formulations are presented in conservative form and for smooth solutions also in the corresponding non-conservative form. 
This allows to obtain stabilization results, when the system is relaxed to a first-order model. 
Computational tests illustrate the theoretical
results. 

\end{abstract}


\section{Introduction}

Nowadays traffic models have become an indispensable tool in the urban and extraurban management of vehicular traffic.
Understanding and developing an optimal transport network, with efficient
movement of traffic and minimal traffic congestions, will have a great socio-economical impact on the society, in particular in pandemics situations. 

{
Besides guaranteeing optimal transport in the presence of pandemic situations, there is a second major aspect, where 
our work on traffic flow modelling may contribute. It is clear that in a pandemic situation the spreading of possible infections correlates with the number of contacts as e.g.~modelled in SIR dynamics~\cite{Kermack1991,Noble1974}. Traffic flow provides valuable information on possible contacts and on possible points of high population density in urban and extraurban areas. The prediction of the flow into and from those areas can help to calibrate the transmission coefficients in typical SIR models for disease propagation. Here, however, deterministic predictions are of little to no use in an a priori assessment of possible critical points of high traffic density. Therefore, it is mandatory to expand the current theory on macroscopic deterministic traffic flow models towards realistic but uncertain models. The current paper precisely tackles this point.}

A vast amount of literature about vehicular traffic modeling has flourished in the last decades. Nevertheless, there are still several limitations for obtaining trustful traffic forecasts. This is possibly due to the fact that the evolution of traffic is described by highly nonlinear dynamics that is also exposed to the presence of various sources and types of uncertainties~\cite{Guiseppe2016,Mattia2020,Tosin2018,Tosin2021}. For example, the uncertainty may stem from real data affected by errors in the measurements or the reaction time of drivers. A pandemic scenario adds additional uncertainties, but needs reliable estimates. 
{
In particular, in view of the discussion of possible measures to reduce traffic and accumulation in certain areas, the reliable and quantifiable prediction is of high importance. The approach presented in this paper allows to quantify the complete statistics of the  uncertain solution and hence it also allows to compute e.g.~rare events.}
Quantifying the propagation of uncertainty in nonlinear models is therefore of interest and the purpose of this paper.
 
Uncertainty quantification in the sense used here is concerned with the propagation of input uncertainty through traffic models. 
Several approaches are presented in the literature and can be classified in non-intrusive and intrusive methods. The main idea underlying the former approach is to solve the model for fixed number of samples using deterministic numerical algorithms. Then, the statistics of the quantities of interest are determined by numerical quadrature. Typical examples are Monte-Carlo and stochastic collocation methods \cite{S4}. 

In contrast, we consider the intrusive stochastic Galerkin method. Here,
stochastic processes are represented as piecewise orthogonal functions, for instance Legendre polynomials or multiwavelets.
These representations are known as generalized polynomial chaos~(gPC) expansions~\cite{S2,S21,Gottlieb2001,S1,S3}. 
Expansions of the stochastic input are substituted into the governing equations and a Galerkin projection is used to obtain deterministic evolution equations for the coefficients of the series expansions.

Results for nonlinear hyperbolic systems are only partial, 
since desired properties like hyperbolicity are not necessarily transferred to the intrusive formulation~\cite{H0,JinShu2018}.  A problem is posed by the fact that the deterministic Jacobian of the projected system differs from the random Jacobian of the original system. We refer the interested reader to~\cite{FettesPaper} for examples of the Euler as well as shallow water equations. Furthermore, it is remarked in~\cite[Sec.~5]{S5} 
that simulations for Euler equations  may break down for high Mach numbers unless auxiliary variables and wavelet-based expansions are used. 

Still, stochastic Galerkin methods applied to hyperbolic equations is an active field of research. {
Those can be successfully applied to scalar conservation laws,  since the resulting Jacobian is symmetric. 
In the scalar case,  well-balanced schemes have been developed~\cite{H5} and a maximum-principle can be ensured~\cite{kusch2019maximum}.}

Furthermore, entropy-entropy flux pairs and hence hyperbolicity can be transferred to a stochastic Galerkin formulation by introducing auxiliary variables~\cite{H0}, which require \emph{expensive} variable transforms. 
Although there are many attempts to make the transform  more efficient and stable~\cite{kusch2020filtered,kusch2020intrusive}, the computational cost remain a drawback of this approach. 
To this end, an expansion in Roe variables has been proposed~\cite{S5}.  Since it exploits quadratic relationships, the necessary transforms are numerically cheap and stable. These auxiliary variables enable also a hyperbolic stochastic Galerkin formulation for isothermal Euler equations for arbitrary gPC expansions.  {Moreover, it has been observed that the shallow water equations  allow for a hyperbolic stochastic Galerkin formulation which neither requires  auxiliary variables nor any transform~\cite{dai2020hyperbolicity}.}

Additional results are available for certain wavelet-based gPC expansions, including the Wiener-Haar basis and piecewise linear multiwavelets~\cite{LEMAITRE2004,S5}.  
These wavelet expansions are motivated by a robust expansion for solutions that depend on the stochastic input in a non-smooth way and are used for stochastic multiresolution as well as adaptivity in the stochastic space~\cite{SI1,Kroeker2015,Tryoen2012}.

In this paper, we consider hyperbolic systems used in vehicular traffic modeling, namely second order macroscopic models \cite{aw2000SIAP, zhang2002non}. The main feature is that they take into account the non-equilibria states, assuming that accelerations are not instantaneous. 
They are able to recover typical traffic phenomena as generating capacity drop, hysteresis, relaxation, platoon diffusion, or spontaneous congestions like stop-and-go waves~\cite{Greenberg2004,Rosales2009,Guiseppe2021}.

The first results in this direction were proposed by Payne and Whitham~\cite{whitham1974linear} taking into account that the speed of each car does not change instantaneously. However, their model has the drawback that the driver's decision is influenced by the road conditions behind. A second order model is due to Aw, Rascle ~\cite{aw2000SIAP} and Zhang ~\cite{zhang2002non}. By taking into account the differences between traffic and fluid flows, they designed models  to simulate the anisotropic traffic behaviour. 

The inhomogeneous Aw-Rascle-Zhang (ARZ) model 
includes a relaxation term that allows drivers 
to achieve the equilibrium speed~\cite{Greemberg2002}. 
In the small relaxation limit the ARZ model approaches to the Lighthill-Whitham-Richards (LWR) model \cite{LWR,richards1956OR}, which can be obtained by means of a  Chapman-Enskog-type expansion. 
 Here, the stability and well-posedness of solutions to the hyperbolic ARZ model is governed by the study of the sign of the diffusion coefficient, which requires the so-called sub-characteristic condition~\cite{Chen1994,Jin1995}. 
The diffusion term vanishes in the zero-relaxation limit and  the LWR model is recovered~\cite{Seibold2013,herty2020bgk,Guiseppe2021}.

This paper analyzes stochastic Galerkin formulations for the  Aw-Rascle-Zhang model in conservative and non-conservative form. The non-conservative form allows to state eigenvalues  and hence ensures hyperbolicity. 
Furthermore, the stability of the system is investigated if it is relaxed to a first-order model. 
As basic tool we follow the approach in~\cite{Seibold2013,herty2020bgk,Guiseppe2021}  and study definiteness properties of the corresponding diffusion coefficient by using a Chapman-Enskog-type expansion. 

Section~\ref{sec:ARZ} introduces the deterministic 
Aw-Rascle-Zhang model in conservative and non-conservative form. 
Section~\ref{sec:state} presents stochastic Galerkin formulations. 
For a \emph{special class} of wavelet-based gPC expansions 
an auxiliary variable that \emph{does not} cause any computationally expensive transforms is introduced to ensure hyperbolicity.  
Section~\ref{sec:inhomo} is devoted to a 
stability analysis of the inhomogeneous ARZ model. 
The theoretical results are derived \emph{only for classical smooth solutions with deterministic relaxation}. 
Riemann problems to weak solutions with uncertainties in the relaxation parameter are illustrated numerically in Section~\ref{sec:num}.

\section{Second order traffic flow models with relaxation}\label{sec:ARZ}

Typical macroscopic traffic flow models describe the density $\rho=\rho(t,x)$ 
and the mean velocity ${v=v(t,x)}$ of vehicles at a location  $x\in \R$ and  time $t>0$.
The natural assumption that the total mass is conserved leads to impose that the density $\rho$ satisfies the continuity equation
\begin{equation}\label{eq:LWR}
\partial_t \rho + \partial_x(\rho v)=0
\quad\text{with initial values} \quad \rho(0,x)=\rho_0(x).
\end{equation}
In first-order models the velocity $v=v(\rho)$ is given as a function of the density alone, e.g. the LWR model \cite{LWR,richards1956OR}. 
Second-order models describe the velocity by an additional differential equation. In particular, we consider the \textbf{inhomogeneous Aw-Rascle-Zhang model}~\cite{aw2000SIAP,Greemberg2002} with relaxation
\begin{equation}\label{eq:arz_non_cons}
\begin{cases}
\begin{aligned}
\partial_t \rho + \partial_x(\rho v)&=0,\\
\partial_t \big(v+\hesitation(\rho)\big) + v \partial_x\big(v+\hesitation(\rho)\big)
&=
\frac{1}{\tau}
\big(\Veq(\rho)-v\big). 
\end{aligned}
\end{cases}
\end{equation}
Here,  $\hesitation(\rho):\R^+\to\R^+$ is called  hesitation or traffic pressure \cite{fan2013comparative}. It is a smooth, strictly increasing  function of the density. 
The relaxation term with parameter~${\tau>0}$ on the right hand side makes the drivers tend to a given equilibrium velocity~$\Veq(\rho)$. 
This is important, since the  homogeneous ARZ model without relaxation has no mechanism to
move drivers when initially are at rest. 
By introducing the variable $z=\rho\big(v+\hesitation(\rho)\big)$, 
the system~\eqref{eq:arz_non_cons} can be written in  conservative form as
\begin{equation}\label{eq:inhomo}
\begin{cases}
\begin{aligned}
\partial_t \rho + \partial_x\big(z-\rho \hesitation(\rho) \big)&=0,\\
\partial_t z + \partial_x\big( \nicefrac{z^2}{\rho}-z \hesitation(\rho) \big)
&=
\frac{\rho}{\tau}
\big(\Veq(\rho)-v(\rho,z)\big)
\end{aligned}
\end{cases}
\quad\text{for}\qquad
v(\rho,z)=\nicefrac{z}{\rho}-\hesitation(\rho). 
\end{equation}
Here, the velocity~$v(\rho,z)$ is a driver dependent property. The conservative formulation~\eqref{eq:inhomo} is abbreviated as 
\begin{align*}
&\partial_t \u + \partial_x \f(\u) = \frac{1}{\tau} \mathbf{S}(\u)
\quad \text{with unknowns} \quad
\u=\binom{\rho}{z}
\quad\text{and} \\
&\f(\u)
=
\binom{\f_\rho(\rho,z)}{\f_z(\rho,z)}
=
\binom{z-\rho \hesitation(\rho)}{
 \nicefrac{z^2}{\rho}-z \hesitation(\rho)
},\quad
\mathbf{S}(\u)
=
\binom{0}{\mathbf{S}_z(\rho,z)}
=
\binom{0}{
\rho\big( \Veq(\rho)-v(\rho,z) \big)
}.
\end{align*} 
The eigenvalues of the Jacobian~
\begin{equation}\label{DeterministicJacobian}
\D_{\u} \f(\u)
=
\Big(\partial_\alpha \f_\beta(\u) \Big)_{\alpha,\beta\in\{\rho,z\}}
=
\begin{pmatrix}
- \hesitation(\rho) -\rho \hesitation'(\rho) & 1 \\
 - \big( \frac{z}{\rho} \big)^2 - z \hesitation'(\rho)
 & 2 \frac{z}{\rho}  - \hesitation(\rho)
\end{pmatrix}
\end{equation}
 are $\lambda_1(\rho,z)=v(\rho,z)-\rho \hesitation'(\rho)$ and $\lambda_2(\rho,z)=v(\rho,z)$. 
Hence, the ARZ model is strictly hyperbolic under the assumption $\rho>0$. 
The (local)  equilibrium velocity~$\Veq(\rho)$ satisfies the scalar conservation law
\begin{equation}\label{equilbriumEquation}
\partial_t \rho
+
\partial_x 
\f_{\textup{eq}}(\rho)=0
\quad\text{for}\quad
\f_{\textup{eq}}(\rho) = \rho \Veq(\rho)
\quad\text{and}\quad
\f_{\textup{eq}}'(\rho) =  \Veq(\rho) + \rho \Veq'(\rho).
\end{equation}

\noindent
Stability requires that the full system propagates information faster than the local equilibrium, i.e.~the~\textbf{sub-characteristic condition}
\begin{equation}\label{SC}\tag{\textup{SC}}
\lambda_1 \Big( \rho,\rho\big( \Veq(\rho) +h(\rho) \big) \Big)
\leq
\f_{\textup{eq}}'(\rho)
\leq 
\lambda_2 \Big( \rho,\rho\big( \Veq(\rho) +h(\rho) \big) \Big)
\quad\text{with}\quad \Veq'(\rho)<0
\end{equation}
is satisfied. 
It is shown in~\cite[Th.~3.1]{Chen1994} for general 2$\times$2 systems that the sub-characteristic condition holds if and only if the 
 \textbf{first-order correction} 
$$
v = \Veq(\rho) + \tau v^{(1)}
+
\mathcal{O}\big(\tau^2\big)
$$
leads to a dissipative advection-diffusion equation. For the deterministic ARZ model~\cite{zhang2002non,herty2020bgk}, this reads as
\begin{equation}\label{DI}\tag{\textup{DI}}
\partial_t \rho 
+
\partial_x \f_{\textup{eq}}(\rho)
=
\tau\partial_x
\big(
\mu(\rho)
\partial_x \rho
\big)
\quad
\text{with diffusion coefficient}\quad
\mu(\rho)
\coloneqq
-\,
\rho^2  \Veq'(\rho) \big( \Veq'(\rho) + h'(\rho) \big).
\end{equation}

\noindent
In the sequel, we will extend these results to the stochastic case.

\section{Stochastic Galerkin formulation}\label{sec:state}

We extend the hyperbolic balance law~\eqref{eq:inhomo} to account for uncertainties that arise from random initial conditions. The hesitation function and the equilibrium velocity, however, remain given deterministic functions. 
Uncertainties are summarized in a random variable~$\xi$, 
defined on a probability space $\big(\Omega,\mathcal{F}(\Omega),\mathbb{P}\big)$,
 and propagated by the random system
\begin{equation}\label{eq:in_system}
    \partial_t \u(t,x,\xi) + \partial_x \f\big(\u(t,x,\xi)\big)=
    \frac{1}{\tau}
    \mathbf{S}\big(\u(t,x,\xi)\big). 
\end{equation}
{
For fixed time and space coordinates we expand the solution in terms of the~\textbf{generalized polynomial chaos (gPC) expansion}}
\begin{equation*}
\label{gPC} \tag{\textup{gPC}}
\G_K[\u](t,x,\xi) \coloneqq \sum_{k=0}^{K} \uhat_k(t,x) \phi_k(\xi)
\quad\text{with gPC modes}\quad
\uhat
\coloneqq
\begin{pmatrix}
\hat{\rho} \\ \hat{z}
\end{pmatrix}
\in\mathbb{R}^{2(K+1)}.
\end{equation*}
The piecewise polynomial functions $\phi_k(\xi)$ form an orthonormal basis with respect to the weighted inner product
$$
\big\langle \phi_i(\cdot), \phi_j(\cdot) \big\rangle
=
\int \phi_i(\xi) \phi_j(\xi) \d \mathbb{P}
=
\delta_{i,j}.
$$
If the random solution~$\u(t,x,\xi)$ is known, the gPC modes can be determined by the orthonormal projection
$ \big\langle \u(t,x,\cdot), \phi_k(\cdot) \big\rangle$. 
Under mild conditions on the probability measure the truncated expansion~\eqref{gPC} converges in the sense~${
	\big\lVert \G_K[\u](t,x,\cdot) - \u(t,x,\cdot) \big\rVert \rightarrow 0
}$
for~${
	K \rightarrow \infty
}$ \cite{S2,funaro2008,Ullmann2012}.

A challenge occurs, since only the gPC modes~$\uhat(0,x)$ corresponding to the initial data are known. To determine them for $t>0$, we derive a differential equation, called stochastic Galerkin formulation,  that describes their propagation in time and space.

\subsection{A semi-intrusive approach as introductory example}\label{SemiIntrusive}
A naive approach  would be
to substitute the truncated expansion~$\eqref{gPC}$ into the random system~\eqref{eq:in_system} 
and then use a Galerkin ansatz to project it onto the space spanned by the basis functions. The resulting system, without relaxation term,  reads as~$
\partial_t \uhat
+
\partial_x \fhat(\uhat)=\veco
$ 
for $\veco\in\mathbb{R}^{K+1}$ 
\begin{align}
&\text{with flux function}
 &\fhat\big(\uhat(t,x)\big)
  =&
 \Bigg\langle
 \f\bigg(  
 \sum\limits_{k=0}^K
 \uhat_k(t,x)\phi_k(\cdot)
 \bigg),\phi_i(\cdot)
 \Bigg\rangle_{i=0,\ldots,K} \label{useless1} \\
&\text{and Jacobian}
&\D_{\uhat} \fhat\big(\uhat(t,x)\big)
  =&
\begin{pmatrix}
\fhat_{{\rho},{\rho}}\big(\uhat(t,x)\big) & \fhat_{{\rho},{z}}\big(\uhat(t,x)\big) \\
\fhat_{{z},{\rho}}\big(\uhat(t,x)\big) & \fhat_{{z},{z}}\big(\uhat(t,x)\big) 
\end{pmatrix} \label{useless2} \\
&\text{consisting of block matrices}
&\fhat_{{\alpha},{\beta}}
\big(\uhat(t,x)\big)
=&
\Bigg\langle 
\partial_{{\alpha}} \f_{{\beta}} 
\bigg( \sum_{k=0}^K \u_k(t,x) \phi_k(\cdot) \bigg), \phi_i(\cdot) \phi_j(\cdot) 
\Bigg\rangle_{i,j=0,\ldots,K}.  \nonumber
\end{align}
Here, the Jacobian~$\D_{\uhat}  \fhat(\uhat)$  consists of the projected entries of the deterministic Jacobian~\eqref{DeterministicJacobian}. 
The Jacobian~\eqref{useless2}, however, has not necessarily real eigenvalues and a full set of eigenvectors. 
In the case of the Aw-Rascle-Zhang model, the flux function~\eqref{useless1} and its Jacobian~\eqref{useless2} are not even directly specifiable, since the deterministic expressions \eqref{eq:inhomo} and  \eqref{DeterministicJacobian} envolve the terms~$\nicefrac{z^2}{\rho}$, 
$\nicefrac{z}{\rho}$ and the possibly nonpolynomial hesitation function~$\hesitation(\rho)$. 
Computing numerically the integrals in equation~\eqref{useless1} and~\eqref{useless2} would lead to an expensive, non-hyperbolic, semi-intrusive scheme. 

\subsection{Intrusive formulation for general gPC expansions}

Instead, we follow the approaches~\cite{JinShu2018,dai2020hyperbolicity} to handle the terms ~$\nicefrac{z^2}{\rho}$ and  
$\nicefrac{z}{\rho}$. 
{
 We introduce the Riemann invariant $\velocity\coloneqq \nicefrac{z}{\rho}$ in the original ARZ model~\cite{aw2000SIAP}. }
While the semi-intrusive approach in Section~\ref{SemiIntrusive} computes the gPC modes~$\velocityhat$ by the orthonormal projection~$ \langle \velocity,\phi_k\rangle$, we project the product~$\rho \velocity $ and determine the modes by the pseudo-spectral Galerkin product 
$
\big\langle
\G_K[ \rho ]
\G_K[ \velocity ],
\phi_k
\big\rangle
=
\hat{z}_k. 
$
Similarly to~\cite{S15,S4,S18}, we express it by
\begin{equation}\label{GP}
\hat{\rho}
\ast
\velocityhat
\coloneqq
\P(\hat{\rho})
\velocityhat
=
\hat{z}\in\mathbb{R}^{K+1}
\quad\text{for}\quad
 \P(\hat{\rho})\coloneqq\sum\limits_{k=0}^K \hat{\rho}_k\mathcal{M}_k
 \quad\text{and}\quad
\mathcal{M}_k\coloneqq\langle \phi_k,\phi_i\phi_j \rangle_{i,j=0,\ldots,K}.
\end{equation}
The matrix~$\P(\hat{\rho})$ is strictly positive definite and hence invertible provided that the gPC expansion~$\G_K[ \hat{\rho} ]>0$ is strictly positive~\cite{Sonday2011,H8,FettesPaper}. \emph{The strict  positive definiteness of the matrix $\P(\hat{\rho})$ is assumed throughout this paper. This assumption excludes vacuum states.}   
We have for the inverse terms 
the pseudo-spectral gPC approximations 
$
\velocityhat
=
\P^{-1}(\hat{\rho}) \hat{z}
$ 
and 
$\hat{z}\ast \velocityhat $,  i.e.
$$
\bigg\lVert
\frac{z^2(\xi)}{\rho(\xi)}
-
\sum\limits_{k=0}^K
\big(
\hat{z}\ast \velocityhat
\big)_k
\phi_k(\xi)
\bigg\rVert
\rightarrow
0
\quad\text{and}\quad
\bigg\lVert
\frac{z(\xi)}{\rho(\xi)}
-
\sum\limits_{k=0}^K
 \velocityhat_k
\phi_k(\xi)
\bigg\rVert
\rightarrow
0
\quad\text{for}\quad
K\rightarrow\infty.
$$
This yields for general gPC bases a stochastic Galerkin formulation for the homogeneous ARZ model, without relaxation, as
\begin{equation}\label{gPCgeneral}
\begin{cases}
	\begin{aligned}
		\partial_t \hat{\rho} + \partial_x\Big(\hat{z}-\hat{\rho}\ast \hesitationhat(\hat{\rho}) \Big)&=\veco,\\
		\partial_t \hat{z} + \partial_x\Big( \hat{z}\ast \big( \P^{-1}(\hat{\rho}) \hat{z}\big) - \hat{z}\ast \hesitationhat(\hat{\rho}) \Big)
		&=\veco,
	\end{aligned}
\end{cases}
\end{equation}
where $\hesitationhat(\hat{\rho}) \in \mathbb{R}^{K+1} $ denotes a given gPC formulation of a hesitation function.  For example, the linear hesitation function  $h({\rho})=\rho$ has the gPC modes  $\hesitationhat(\hat{\rho})=\hat{\rho}$.
By using the following calculation rules, see e.g.~\cite{S5,JinShu2018,GersterHertyCicip2020},
\begin{equation}\label{eq:G_sym}
\hat{\rho} \ast \hat{z} = \hat{z} \ast \hat{\rho},
\quad
\D_{\hat{\rho}} \big[ \hat{\rho} \ast \hat{z} \big] 
=
\P(\hat{z}),
\quad
\D_{\hat{\rho}} \big[ \P^{-1} (\hat{\rho} )\hat{z} \big]
=
- 
 \P^{-1} (\hat{\rho} )\
\P
 \big(
\P^{-1} (\hat{\rho} )\hat{z} \big)
\end{equation}
we obtain the Jacobian of the gPC formulation~\eqref{gPCgeneral} as 
\begin{equation*}
\D_{\uhat} \fhat (\uhat )
=
\begin{pmatrix}
-\P\big( \hesitationhat(\hat{\rho}) \big)
-\P(\hat{\rho}) \hesitationhat'(\hat{\rho})
& \indikator \\
-
\P(\hat{z}) \P^{-1}(\hat{\rho}) 
\P( \P^{-1}(\hat{\rho}) \hat{z})
-
\P(\hat{z}) \hesitationhat'(\hat{\rho}) 
&
\P(\hat{z}) \P^{-1}(\hat{\rho}) 
+
\P\big(\P^{-1}(\hat{\rho}) \hat{z})
-\P\big( \hesitationhat(\hat{\rho}) \big)
\end{pmatrix},
\end{equation*}
where~$\indikator\coloneqq \diag\{1,\ldots,1\}$ denotes the identity matrix. 
The matrices~$\M_k$ and hence the linear operator~$\P\, :\, \mathbb{R}^{K+1} \rightarrow \mathbb{R}^{(K+1) \times(K+1) } $, defined in equation~\eqref{GP}, are \emph{exactly} computable in an offline stage. Therefore, the stochastic Galerkin formulation~\eqref{gPCgeneral} is intrusive and no numerical quadrature is needed during a simulation. Furthermore, the eigenvalues can be exactly computed. However, eigenvalues are not proven real which motivates the following subsection.

\subsection{Hyperbolic and intrusive formulation for wavelet-based gPC expansions} \label{sec:cons}
Under additional assumptions on the bases functions, hyperbolicity can be guaranteed.  We consider basis functions~$\phi_k$ that satisfy the following properties:\\

	\textup{(A1)}\quad 
	The precomputed matrices $\mathcal{M}_\ell$ and $\mathcal{M}_k$ commute for all ${\ell,k = 0,\ldots,K}$.
	
	\textup{(A2)}\quad
	There is an eigenvalue decomposition~${
	\P(\widehat{\alpha}) = V \mathcal{D}(\widehat{\alpha}) V^{\T}
	}$ 
	with constant eigenvectors.	
	
	\textup{(A3)}\quad
	The matrices~$\P(\widehat{\alpha})$ and $\P(\widehat{\beta})$ commute for all ${\widehat{\alpha}, \widehat{\beta} \in \mathbb{R}^{K+1}}$.\\

These properties  have been proven equivalent in~\cite[Lem.~4.1]{GersterHertyCicip2020}. 
Property~(A1) allows for a numerical verification in a precomputation step such that basis functions satisfy also the other properties, which may be difficult to prove analytically. 
Property~(A2) has been shown directly for the Wiener-Haar basis in~\cite[Appendix~B]{S5}, which we will consider in~Section~\ref{sec:num}. 
This property allows for an efficient numerical implementation, since the eigenvalues~$
\mathcal{D}(\widehat{\alpha}) 
=
V^{\T}
\P(\widehat{\alpha})  
V
$ are directly computable by a numerically cheap and stable matrix multiplication. 
Property (A3) has a technical benefit, needed for the following theoretical results. 
{
Following~\cite{GersterHertyCicip2020,StephansDiss},  polynomial functions~$\hesitation(\rho)=\rho^\gamma$, $\gamma\in\mathbb{N}$ 
 and their Jacobians are expressed as 
\begin{equation}\label{hesitationFunction}
\hesitationhat (\hat{\rho})
\coloneqq
\P^{\gamma-1}(\hat{\rho}) \hat{\rho}
=
V\mathcal{D}(\hat{\rho})^{\gamma-1} V^\T \hat{\rho}, \quad
\hesitationJacobi(\hat{\rho})
=
\D_{\hat{\rho}}
\hesitationhat(\hat{\rho})
=
\gamma \P^{\gamma-1}(\hat{\rho})
=
\gamma V\mathcal{D}(\hat{\rho})^{\gamma-1} V^\T. 
\end{equation} 
Furthermore, the equality 
$
\P\big( \P^{-1} (\widehat{\rho}) \widehat{z}\big)=\P^{-1}(\widehat{\rho})\P(\widehat{z})
$ is satisfied provided that 
properties~(A1)~--~(A3) hold. 
Equation~\eqref{hesitationFunction} and~\cite[Remark~1]{FettesPaper}, where the representation of nonpolynomial functions is discussed,  
motivate to assume  possibly nonpolynomial hesitation functions for~$\gamma\geq 1$ and a Jacobian of the form~$
\hesitationJacobi(\hat{\rho})=V\Dh(\hat{\rho}) V^\T 
$ 
with strictly positive eigenvalues~$\Dh(\hat{\rho})>0$.} 
Under these assumptions, we have  the stochastic Galerkin formulation 
$
\partial_t \uhat
+
\partial_x \fhat(\uhat)=\veco
$ 
for the homogeneous ARZ model
\begin{align}
&\text{with flux function}
&\fhat (\uhat)
&=
\begin{pmatrix}
\hat{z}
-
\P(\hat{\rho}) \hesitationhat(\hat{\rho}) \\
\P( \hat{z} ) \P^{-1}(\hat{\rho}) \hat{z}
 -
\P(\hat{z}) \hesitationhat(\hat{\rho})
\end{pmatrix} \label{flux_AwRascle} \\
&\text{and Jacobian}\quad
&\D_{\uhat} \fhat (\uhat)
&=
\begin{pmatrix}
-  \P\big(\hesitationhat (\hat{\rho} )\big)- \P(\hat{\rho}) \hesitationJacobi(\hat{\rho}) & \indikator \\
-
\P^2(\hat{z}) \P^{-2}(\hat{\rho})
-
\P(\hat{z}) \hesitationJacobi(\hat{\rho})
&
2 \P(\hat{z}) \P^{-1}(\hat{\rho}) 
-
\P\big( \hesitationhat(\hat{\rho}) \big)
\end{pmatrix}. \nonumber
\end{align}


\subsection{Stochastic Galerkin formulation for the inhomogeneous ARZ model}

The hyperbolic formulation, presented in Subsection~\ref{sec:cons}, is directly extendable to a stochastic Galerkin formulation for the inhomogeneous ARZ model. 
To this end, we assume an \emph{arbitrary, but consistent} gPC expansion~$\widehat{\Veq}(\hat{\rho})$ of the random equilibrium speed~$\Veq\big(\rho(\xi)\big)$, satisfying
$$
\bigg\lVert
\Veq\big(\rho(\xi)\big)
-
\sum\limits_{k=0}^K
\widehat{\Veq}(\hat{\rho})_k
\phi_k(\xi)
\bigg\rVert
\rightarrow
0
\quad\text{for}\quad
K\rightarrow\infty.
$$
Then, we introduce a stochastic Galerkin formulation of the relaxation term in the conservative formulation~\eqref{eq:inhomo}
by
\begin{equation}\label{ST}
\S_{\hat{z}} (\uhat)
\coloneqq
\hat{\rho}\ast
\Big(
\widehat{\Veq}(\hat{\rho}) 
-
\hat{v}(\hat{\rho},\hat{z})
\Big)
\quad
\text{with auxiliary variable}
\quad
\hat{v}(\hat{\rho},\hat{z})
=
\P^{-1}(\hat{\rho})  \hat{z} - \hesitationhat(\hat{\rho}).
\end{equation}
This auxiliary variable  
also allows to obtain a stochastic Galerkin formulation for the non-conservative formulation~\eqref{eq:arz_non_cons}. Altogether we have the \textbf{hyperbolic stochastic Galerkin formulations for the inhomogeneous ARZ model} in a 
\begin{align}
&\text{conservative form}
&& 
\begin{cases}
\begin{aligned}
\partial_t \hat{\rho} + \partial_x\Big(\hat{z}-\P(\hat{\rho}) \hesitationhat(\hat{\rho}) \Big)&=\veco,\\
\partial_t \hat{z} + \partial_x\Big( \P(\hat{z})  \P^{-1}(\hat{\rho}) \hat{z} - \P(\hat{z}) \hesitationhat(\hat{\rho}) \Big)
&=\frac{\hat{\rho}}{\tau}
\ast
\Big(
\widehat{\Veq}(\hat{\rho}) 
-
\hat{v}(\hat{\rho},\hat{z})
\Big),
\end{aligned}
\end{cases}
\label{C}\tag{$\mathcal{C}$}
\\
&\text{non-conservative form}
&&
\begin{cases}
\begin{aligned}
\partial_t \hat{\rho} + \partial_x\Big(\P(\hat{\rho}) \hat{v}\Big)&=\veco,\\
\partial_t \Big( \hat{v}+\hesitationhat(\hat{\rho}) \Big) + \P( \hat{v} ) \, \partial_x\Big(\hat{v}+\hesitationhat (\hat{\rho}) \Big)
&=
\frac{1}{\tau}
\Big( \widehat{\Veq}(\hat{\rho})-\hat{v}\Big). 
\end{aligned}
\end{cases}
\label{N}\tag{$\mathcal{N}$}
\end{align}


We show in Theorem~\ref{LEMMA_1} that these two formulations are equivalent for smooth solutions, as it holds in the deterministic case~\cite{aw2000SIAP}. 
However, if there is a jump in the solution, the non-conservative form  contains the product of the discontinuous matrix-valued
function~$\P(\hat{v})$ 
with the distributional derivative of the term 
$
\hat{v}+\hesitationhat(\hat{\rho})
$, 
 which may contain a
Dirac mass at the point of the jump. In general, such a product is not well-defined~\cite[Sec.~1]{BRESSAN}. 
Theorem~\ref{LEMMA_1} ensures that the system is strongly hyperbolic, which means that eigenvalues of the Jacobian~$\D_{\uhat} \fhat (\uhat)$, i.e.~the characteristic speeds of the hyperbolic system are real. Moreover, the Jacobian~$\D_{\uhat} \fhat (\uhat)$  admits a complete set of eigenvectors which implies that classical solutions are well-posed~\cite{Kreiss2013}. 

\begin{theorem}\label{LEMMA_1}
Let a gPC expansion with the properties~(A1)~--~(A3), 
a stochastic Galerkin formulation of a hesitation function~$\hesitationhat(\hat{\rho})$ and  a Galerkin formulation of an equilibrium  velocity~$\widehat{\Veq}(\hat{\rho})$ be given. 
Assume further a Jacobian of the hesitation function
\begin{equation*}
\widehat{\hesitation'}(\hat{\rho}) \coloneqq \D_{\hat{\rho}} \hesitationhat(\hat{\rho}) 
=
V \Dh (\hat{\rho}) V^\T 
\end{equation*}
with constant eigenvectors. 
Then, for smooth solutions the conservative~\eqref{C} and non-conservative~\eqref{N} stochastic Galerkin formulations to the inhomogeneous ARZ model are equivalent. The characteristic speeds are
$$
\widehat{\lambda_1}(\hat{\rho},\hat{z}) =
\mathcal{D}\big(\hat{v}(\hat{\rho},\hat{z}) \big) 
-
\Dh({\hat{\rho}} )
\mathcal{D}{(\hat{\rho}}) 
\quad\text{and}\quad
\widehat{\lambda_2}(\hat{\rho},\hat{z})=\mathcal{D}\big(\hat{v}(\hat{\rho},\hat{z}) \big)
\quad
\text{for}
\quad
\hat{v}(\hat{\rho},\hat{z})
=
\P^{-1}(\hat{\rho})  \hat{z} - \hesitationhat(\hat{\rho}),
$$
where 
$
\mathcal{D}(\hat{v}) 
$ 
denote the eigenvalues of the matrix 
$
\mathcal{P}(\hat{v})$.
Furthermore, the stochastic Galerkin formulations~\eqref{N}~and~\eqref{C} are strongly hyperbolic in the sense that the characteristic speeds are real 
and the Jacobian~$\D_{\uhat} \fhat (\uhat)$ admits a complete set of eigenvectors. 

\end{theorem}

\begin{proof}{
Provided that properties~(A1)~--~(A3) hold, we have~$
	\P\big(
	\P(\hat{v})
	\hat{\rho}
	\big)
	=
	\P(\hat{v})
	\P(
	\hat{\rho}
	)
	$}	and the Galerkin product is symmetric \eqref{eq:G_sym}.
	 Hence, 
we obtain	
$$	
\big(\hat{\rho} \ast \hat{v}  \big) 
\ast
\big( \hat{v}+\hesitationhat (\hat{\rho}) \big)
=
\P\big(
\P(\hat{v})
\hat{\rho}
\big)
\P^{-1}(\hat{\rho})
\hat{z}
=
\P(\hat{v})
\hat{z}
=
\P(\hat{z})
\hat{v}
=
\P(\hat{z})
\P^{-1}(\hat{\rho})
\hat{z}
-
\P(\hat{z}) 
\hesitationhat(\hat{\rho}). 
$$	
Since the opertor~$\P(\hat{\rho})$ is linear,  the homogeneous part of the non-conservative formulation can be rewritten as
\begin{align*}
\veco
&=
\big( \hat{v}+\hesitationhat (\hat{\rho}) \big) \ast \Big[
\partial_t \hat{\rho} + \partial_x\big(\widehat{\rho} \ast \hat{v}\big) \Big]
+\hat{\rho} \ast \Big[
\partial_t \big( \hat{v}+\hesitationhat (\hat{\rho}) \big) +  \hat{v}  \ast \partial_x\big(\hat{v}+\hesitationhat (\hat{\rho}) \big) \Big] \\
&=
\partial_t 
\hat{z}
+
\partial_x
\Big(
\big(\hat{\rho} \ast \hat{v}  \big) 
\ast
\big( \hat{v}+\hesitationhat (\hat{\rho}) \big)
\Big) \\
&=
\partial_t \hat{z}
+
\partial_x
\Big(
\P(\hat{z})
\P^{-1}(\hat{z})
\hat{z}
-
\hat{z} \ast 
\hesitationhat(\hat{\rho})
\Big).
\end{align*}
{
Here, we have used the equality~$
\hat{\rho} \ast (\hat{v}\ast \partial_x ) 
=
(\hat{\rho} \ast \hat{v})\ast \partial_x  
$, which is satisfied provided that the  assumptions~{(A1)~--~(A3)} hold, but not for general gPC bases, since the Galerkin product is typically not associative~\cite{S18,S15}.} 
Likewise, the relaxation term in the conservative formulation is obtained by multiplying~$\P(\hat{\rho})$, i.e.~by applying the Galerkin product to the relaxation term of the non-conservative form. 
Therefore, the two formulations \eqref{C} and \eqref{N} are equivalent. 
To state the eigenvalues, we rewrite the first equation of the non-conservative form as
\begin{equation}
\label{LemmaEq1}
\veco
=
\hesitationJacobi(\hat{\rho})
\Big[
\partial_t \hat{\rho} + \partial_x\big(\widehat{\rho} \ast \hat{v}\big)
\Big]
=
\partial_t
\hesitationhat(\hat{\rho})
+
\hat{v}\ast
\hesitationhat(\hat{\rho})_x
+
\hesitationJacobi(\hat{\rho})
(\hat{\rho} \ast \hat{v}_x),
\end{equation}
where we have used the symmetry of the Galerkin product. 
By subtracting equation~\eqref{LemmaEq1} from the second equation in the non-conservative form and by using~property~(A2), i.e.~an eigenvalue decomposition with  constant, orthonormal eigenvectors~$V^\T=V^{-1}$, we obtain
\begin{align}
&\partial_t
\begin{pmatrix}
\hat{\rho} \\ \hat{v}
\end{pmatrix}
+
\begin{pmatrix}
\P(\hat{v}) & \P(\hat{\rho}) \\
\zeros
& \P( \hat{v} ) - \hesitationJacobi(\hat{\rho}) \P(\hat{\rho}) 
\end{pmatrix}
\partial_x 
\begin{pmatrix}
\hat{\rho} \\ \hat{v}
\end{pmatrix}
=
\veco \nonumber \\
\iff \qquad 
\partial_t&
\begin{pmatrix}
V^\T \hat{\rho} \\ V^\T \hat{v}
\end{pmatrix}
+
\begin{pmatrix}
\mathcal{D}(\hat{v}) & \mathcal{D}(\hat{\rho}) \\
\zeros & \mathcal{D}( \hat{v} ) 
- 
\Dh({\hat{\rho}} )
\mathcal{D}{(\hat{\rho}}) 
\end{pmatrix}
\partial_x 
\begin{pmatrix}
V^\T \hat{\rho} \\ V^\T \hat{v}
\end{pmatrix}
=
\veco \label{SparseMatrix}
\end{align}
for $\veco\in\mathbb{R}^{2(K+1)} $ and 
$\zeros\in\mathbb{R}^{(K+1)\times(K+1)} $. 
Due to the sparsity structure in the quasilinear form~\eqref{SparseMatrix} 
 a complete set of eigenvectors exists and eigenvalues $\hat{\lambda_1}$, $\hat{\lambda_2}$ are obtained.

\end{proof}

\section{Stability analysis of the inhomogeneous ARZ model}\label{sec:inhomo}

	


The parameter~${\tau>0}$ determines the relaxation of the velocity~$\hat{v}(\hat{\rho},\hat{z})$, given by equation~\eqref{ST} as auxiliary variable, towards the gPC modes~$\widehat{\Veq}(\hat{\rho})$ of the equilibrium velocity, which is a function of the density alone.  
We study in this section \emph{small, but positive} values of the relaxation paramter~$\tau>0$, when the ARZ model is close to the  
\begin{align}
&\text{equilibrium model}
&&
\partial_t 
\hat{\rho}
+
\partial_x
\widehat{\f_{\textup{eq}}}(\hat{\rho})
=\veco,\quad
\widehat{\f_{\textup{eq}}}(\hat{\rho})
=
 \hat{\rho} \ast \widehat{\Veq}(\hat{\rho}) 
 \label{modeleq} \\
&\text{with Jacobian}
&&
\D_{\hat{\rho}}
\Big( \hat{\rho} \ast \widehat{\Veq}(\hat{\rho})  
\Big)
=
\P \Big( \widehat{\Veq}(\hat{\rho}) \Big)
+
\P( \hat{\rho} ) 
\D_{\hat{\rho}}
\widehat{\Veq}(\hat{\rho}).
\label{Jacobieq} 
\end{align}
We observe from the Jacobian~\eqref{Jacobieq} that an eigenvalue decomposition of the equilibrium velocity of the form
$$
\widehat{\Veq'}(\hat{\rho})
\coloneqq
\D_{\hat{\rho}}
\widehat{\Veq}(\hat{\rho})
=
V 
\DVeq (\hat{\rho})
V^\T
\quad
\text{with negative eigenvalues}
\quad
\DVeq (\hat{\rho})<\veco
$$
should be assumed such that 
all  waves of the equilibrium model propagate at the {characteristic speeds
$$
\Deq(\hat{\rho})
\coloneqq
\mathcal{D} \Big( \widehat{\Veq}(\hat{\rho}) \Big)
+
\mathcal{D}( \hat{\rho} ) 
\DVeq (\hat{\rho})
$$
not exceeding the equilibrium velocity.} This is identified by the eigenvalues of the matrix $\P\Big(\widehat{\Veq}(\hat{\rho})\Big)$. 
Analogously to the analysis in~\cite{Chen1994,Seibold2013,herty2020bgk,Guiseppe2021}, we use a  Chapman-Enskog-type expansion that allows to study the behaviour of first-order perturbations of the equilibrium velocity. 
This yields a diffusion correction as stated in the following theorem.


\begin{theorem}\label{thm:inhomog}
Let a gPC expansion with the properties~(A1)~--~(A3), 
a stochastic Galerkin formulation of a hesitation function~$\hesitationhat(\hat{\rho})$ and  a Galerkin formulation of an equilibrium  velocity~$\widehat{\Veq}(\hat{\rho})$ be given. 
Assume further that the Jacobians can be written as
$$
\widehat{\Veq'}(\hat{\rho}) \coloneqq \D_{\hat{\rho}} 
\widehat{\Veq}(\hat{\rho}) 
=
V \DVeq (\hat{\rho}) V^\T 
\quad\text{and}\quad
\widehat{\hesitation'}(\hat{\rho}) \coloneqq \D_{\hat{\rho}} \hesitationhat(\hat{\rho}) 
=
V \Dh (\hat{\rho}) V^\T 
$$ 
with constant eigenvectors. 
The first-order correction to the local equilibrium approximation reads
\begin{equation}\label{DIhat}\tag{$\widehat{\textup{DI}}$}
\partial_t \hat{\rho} 
+
\partial_x  \widehat{\f_{\textup{eq}}}(\hat{\rho})
=
\tau
\partial_x
\big(
 \hat{\mu}( \hat{\rho} )
\partial_x\hat{\rho}
\big)
,\quad
\hat{\mu}(\hat{\rho})
\coloneqq
-\,
V \bigg[
\mathcal{D}(\hat{\rho})^2  \DVeq(\hat{\rho}) \Big( \DVeq(\hat{\rho}) + \Dh(\hat{\rho}) \Big)
\bigg] V^\T.
\end{equation}
Furthermore, it is dissipative if and only if the sub-characteristic condition 
\begin{equation}\label{SChat}\tag{$\widehat{\textup{SC}}$}
\widehat{\lambda_1}( \hat{\rho},\hat{z} )
\leq
\Deq(\hat{\rho})
\leq 
\widehat{\lambda_2}( \hat{\rho},\hat{z} )
\quad
\text{holds on}
\quad
\hat{z}=
\hat{\rho} \ast \Big( 
\widehat{\Veq}(\hat{\rho})
+
\hat{h}(\hat{\rho})
\Big)
\quad\text{with}\quad
\DVeq (\hat{\rho})<\veco.
\end{equation}

\end{theorem}

\begin{proof}
We apply a Chapman-Enskog expansion 
$$
\hat{v} =  \widehat{\Veq}(\hat{\rho}) 
+
\tau \vI 
+
\mathcal{O}\big(\tau^2 \big). 
$$
The linearity 
$
\P(\hat{\alpha}+\beta)
=
\P(\hat{\alpha})+\P(\beta)
$
implies
\begin{alignat*}{8}	
&\hat{v}
\ast
\partial_x
\hat{v}
&&=
\widehat{\Veq}(\hat{\rho}) \ast \partial_x \widehat{\Veq}(\hat{\rho})
&&+
\mathcal{O}(\tau)
&&=
\P\Big( \widehat{\Veq}(\hat{\rho}) \Big)  \widehat{\Veq'}(\hat{\rho})
\partial_x \hat{\rho} 
&&+ \mathcal{O}(\tau), \\
&\partial_x \big(
\hat{\rho}
\ast
\hat{v}
\big)
&&=
\partial_x \Big(
\hat{\rho} \ast \widehat{\Veq}(\hat{\rho})
\Big)
&&+
\mathcal{O}(\tau)
&&=
\bigg[
\P\Big( \widehat{\Veq}(\hat{\rho}) \Big) 
+
\P( \hat{\rho}) \widehat{\Veq'}(\hat{\rho}) 
\bigg] \partial_x \hat{\rho}
&&+
\mathcal{O}(\tau).
\end{alignat*}
Hence, in the non-conservative formulation we obtain
$$
-\vI
=
\frac{\widehat{\Veq}(\hat{\rho}) - \hat{v}}{\tau}
+\mathcal{O}(\tau)
=
\partial_t 
\Big(
\widehat{\Veq}(\hat{\rho})+\hat{h}(\hat{\rho}) 
\Big)
+
\widehat{\Veq}(\hat{\rho}) \ast \partial_x 
\Big(
\widehat{\Veq}(\hat{\rho})
+
\hat{h}(\hat{\rho}) 
\Big)
+
\mathcal{O}(\tau).
$$
The symmetry of the Galerkin product and the equilibrium model~\eqref{modeleq} yield 
$$
-\vI
=
\bigg(
\widehat{\Veq'}(\hat{\rho})+\widehat{h'}(\hat{\rho}) 
\bigg)
\bigg(
\partial_t \hat{\rho}
+
\P\Big(
\widehat{\Veq}(\hat{\rho})
\Big) \partial_x\hat{\rho}
\bigg)
+ \mathcal{O}(\tau)
=
\bigg(
\widehat{\Veq'}(\hat{\rho})+\widehat{h'}(\hat{\rho}) 
\bigg)
\P( \hat{\rho} )
\widehat{\Veq'}(\hat{\rho}) 
+ \mathcal{O}(\tau), 
$$
which implies the claim 
$$
\partial_t \hat{\rho} 
+
\partial_x  \widehat{\f_{\textup{eq}}}(\hat{\rho})
=
\tau
\partial_x
\big(
\hat{\mu}( \hat{\rho} )
\partial_x \hat{\rho}
\big)
+\mathcal{O}\big(\tau^2\big).
$$

\end{proof}

Theorem~\ref{thm:inhomog} gives conditions 
{to properly choose} a
hesitation  function~$\hesitation(\rho)$ and an equilibrium velocity~$\Veq(\rho)$. In the deterministic case, various choices have been investiated  to model also 
phantom traffic jams and stop-and-go waves by introducing a \emph{negative} diffusion coefficient~\cite{Greenberg2004,Rosales2009,Guiseppe2021}. 
Here, we investigate states close to the equilibrium and choose a hesitation  function~$\hesitation(\rho)$ and an equilibrium velocity~$\Veq(\rho)$ such that sub-characteristic condition is fulfilled. 
{The following corollary extends a widely used class, which includes the Greenshields flux, 
	see e.g.~\cite{greenshields_study_1935,Greemberg2002,Siebel2005} for the deterministic case, to the derived stochastic Galerkin formulation. 
}

\begin{corollary}\label{Corollary}
Let an equilibrium velocity and a hesitation function of the form 
$$
\Veq(\rho)
=
\frac{v_{\max}}{\rho_{\max}}
\Big(
\rho_{\max}-\rho^\gamma
\Big)
\quad\text{and}\quad
\hesitation(\rho)
=
\Veq(0)
-
\Veq(\rho)
=
\frac{v_{\max}}{\rho_{\max}}
\rho^\gamma
$$
with strictly positive constants $v_{\max}, \rho_{\max},\gamma$
be given. Under the assumptions of Theorem~\ref{LEMMA_1} and Theorem~\ref{thm:inhomog} 
 the sub-characteristic condition~\eqref{SChat} is satisfied for the stochastic Galerkin formulations
$$
\widehat{\Veq}(\hat{\rho})
=
\frac{v_{\max}}{\rho_{\max}}
\bigg(
\rho_{\max} e_1- \P(\hat{\rho})^{\gamma-1}\hat{\rho}
\bigg)
\quad
\text{and}
\quad
\hesitationhat(\hat{\rho})
=
\widehat{\Veq}(\veco)
-
\widehat{\Veq}(\hat{\rho})
=
\frac{v_{\max}}{\rho_{\max}}
\P(\hat{\rho})^{\gamma-1}\hat{\rho}
$$
with unit vector 
$e_1=(1,0,\ldots,0)^\T$.
\end{corollary}

\begin{proof}
Equation~\eqref{hesitationFunction} yields 
the stochastic Galerkin formulations and
$$
\widehat{\Veq'}(\hat{\rho})
=
-\gamma \frac{v_{\max}}{\rho_{\max}}
\P(\hat{\rho})^{\gamma-1}
=
-\widehat{h'}(\hat{\rho}) 
\quad
\Leftrightarrow
\quad
\zeros
=
\widehat{\Veq'}(\hat{\rho})+\widehat{h'}(\hat{\rho}) 
\quad
\Leftrightarrow
\quad
\veco
=
\DVeq(\hat{\rho}) + \Dh(\hat{\rho}). 
$$
The matrices~$
\P(\hat{\rho}) 
$ 
and 
$
\P(\hat{\rho})^{\gamma-1}
=
V
\mathcal{D}(\hat{\rho})^{\gamma-1}
V^\T
$
are strictly positive definite. Hence, the Jacobian~$\widehat{\Veq'}(\hat{\rho})$ is strictly negative definite and we have~$\DVeq(\hat{\rho})<\veco $. 

\end{proof}

\section{Numerical results}\label{sec:num} 
The introduction of the gPC modes~$\hat{v}$ as auxiliary variable also allows for an efficient numerical evaluation of the flux function~\eqref{flux_AwRascle}, the relaxation term~\eqref{ST} and the computation of eigenvalues by the numerically cheap and stable matrix vector multiplications
$$
\begin{aligned}
\mathcal{D}(\hat{\rho})
&=
V^\T \mathcal{P}(\hat{\rho}) V, \\
\mathcal{D}\big(\hat{v}(\hat{\rho},\hat{z})\big)
&=
V^\T \mathcal{P}\big(\hat{v}( \hat{\rho},\hat{z} )\big) V, \\
\Dh(\hat{\rho})
&=
V^\T\, \widehat{\hesitation'}(\hat{\rho}) V,
\end{aligned}\qquad
\begin{aligned}
\widehat{\lambda_2}(\hat{\rho},\hat{z})&=\mathcal{D}\big(\hat{v}(\hat{\rho},\hat{z}) \big), \\
\widehat{\lambda_1}(\hat{\rho},\hat{z}) &=
\widehat{\lambda_2}(\hat{\rho},\hat{z})
-
\Dh({\hat{\rho}} )
\mathcal{D}{(\hat{\rho}}) , \\
\hat{v}(\hat{\rho},\hat{z})
&=
V \mathcal{D}^{-1} (\hat{\rho}) V^\T  \hat{z} - \hesitationhat(\hat{\rho}),
\end{aligned}\quad
\begin{aligned}
\fhat (\uhat)
&=
\begin{pmatrix}
\mathcal{P}( \hat{\rho} )  \hat{v}(\hat{\rho},\hat{z}) \\
\mathcal{P}(\hat{z} )  \hat{v}(\hat{\rho},\hat{z}) 
\end{pmatrix}, \\
\S_{\hat{z}} (\uhat)
&=
\mathcal{P}( \hat{\rho} )
\Big(
\widehat{\Veq}(\hat{\rho}) 
-
\hat{v}(\hat{\rho},\hat{z})
\Big).
\end{aligned}
$$
Hence, the computational complexity grows like~$K^2$, which  is relatively low compared to approaches with entropy and Roe variables~\cite{S5,FettesPaper}. The price is the restriction to gPC bases that satisfy the assumptions~(A1)~--~(A3). 
Here, we use the Haar sequence~\cite{Haar1910,S4,S16} with level~${J\in\mathbb{N}_0}$ that generates a  gPC basis~$\mathbb{S}_K$  with ${K+1=2^{J+1}}$ elements by 
\begin{align*}
&	\ \,
\mathbb{S}_K \coloneqq \Big\{ 1, \psi(\xi), \psi_{j,k}(\xi) \ \big| \ k=0,\ldots,2^j-1, \ j=1,\ldots,J  \Big\} 
\quad \ \text{for} \\
&	\psi_{j,k}(\xi)
\coloneqq
2^{\nicefrac{j}{2}}\psi\big(2^j \xi - k \big) 
\quad \text{ and } \quad
\psi(\xi) \coloneqq
\begin{cases}
1  & \text{if \ } 0 \leq \xi < \nicefrac{1}{2},  \\
-1 & \text{if \ } \nicefrac{1}{2} \leq \xi < 1,  \\
0  & \text{else.}
\end{cases}
\end{align*}
Using a lexicographical order we identify the gPC basis  
${\phi_0 = 1}$, 
${\phi_1 = \psi}$, 
${\phi_2 = \psi_{1,0}}$, 
${\phi_3 = \psi_{1,1}}$, etc.\\

\noindent
An equidistant space discretization~${\Delta x>0}$ is used to divide the space interval~$[0,x_{\textup{end}}]$ into $N$~cells such that~${\Delta x N = x_{\textup{end}}}$ with centers~${x_j \coloneqq \big(j+\frac{1}{2}\big) \Delta x}$  and edges~${x_{j-{\nicefrac{1}{2}}} \coloneqq j \Delta x}$.
The discrete time steps are denoted by~${t_k\coloneqq k \Delta t}$ for~${k \in \mathbb{N}_0}$. 
Due to the eigenvalue estimates 
$
\big|  \widehat{\lambda_1}(\hat{\rho},\hat{z}) \big|
\leq \big| \widehat{\lambda_2}(\hat{\rho},\hat{z}) \big|
=
\big| \mathcal{D}\big(\hat{v}(\hat{\rho},\hat{z}) \big) \big|
$ 
a local Lax-Friedrichs flux~\cite{Leveque2} is efficiently  evaluated as 
\begin{equation*}
\NumFlux(\caStoConservedLeft,\caStoConservedRight) 
\coloneqq \frac{1}{2} \Big[ \fhat(\caStoConservedLeft) + \fhat(\caStoConservedRight) \Big] + 
\frac{1}{2}  
\max\limits_{j=\boldsymbol{\ell},\boldsymbol{r}} \bigg\{ \Big|
\mathcal{D}\big(\hat{v}( \caStoConserved_j )\big)
\Big| \bigg\}
(\caStoConservedLeft - \caStoConservedRight).
\end{equation*}
For numerical purposes the relaxation term is expressed as
$$
\S_{\hat{z}} (\uhat)
\coloneqq
\hat{\rho}\ast
\Big(
\widehat{\Veq}(\hat{\rho}) 
-
\hat{v}(\hat{\rho},\hat{z})
\Big)
=
\widehat{M}(\hat{\rho})-\hat{z}
\quad\text{for}\quad
\widehat{M}(\hat{\rho})
\coloneqq
\hat{\rho} \ast 
\Big(
\widehat{\Veq}(\hat{\rho}) 
+\hat{h}(\hat{\rho}) 
\Big).
$$
Since the term~$\S_{\hat{z}} (\uhat)$ depends also in the  stochastic Galerkin formulation  on the unknown~$\hat{z}\in\mathbb{R}^{K+1}$ in a linear way, 
 a {first-order} IMEX scheme~\cite{pareschi2005implicit,herty2020bgk,Pieraccini_Puppo2006}, which treats the advection part explicitly and the possibly stiff relaxation implicitly, can be employed:
\begin{align*}
& \caStoConserved_j^{k+1} = \caStoConserved_j^{k} - \frac{\Delta t}{\Delta x} \Big( \NumFlux \big(\boldsymbol{\bar{u}}_j^{(1)},\boldsymbol{\bar{u}}_{j+1}^{(1)} \big)
-
\NumFlux \big(\boldsymbol{\bar{u}}_{j-1}^{(1)},\boldsymbol{\bar{u}}_j^{(1)}\big)
\Big), \\
&\text{with}\quad
\boldsymbol{\bar{u}}_j^k
=
\Big(
\boldsymbol{\bar{\rho}}_j^k, 
\boldsymbol{\bar{z}}_j^k
\Big)^\T,
\quad
\boldsymbol{\bar{u}}_j^{(1)}
=
\Big(
\boldsymbol{\bar{\rho}}_j^{(1)},
\boldsymbol{\bar{z}}_j^{(1)}
\Big)^\T
\quad\text{and}\quad
\begin{cases}
\begin{aligned}
\boldsymbol{\bar{\rho}}_j^{(1)}
&=
\boldsymbol{\bar{\rho}}_j^k, \\
\boldsymbol{\bar{z}}_j^{(1)}
&=
\frac{\tau}{\tau + \Delta t}
\boldsymbol{\bar{z}}_j^k
+
\frac{\Delta t}{\tau + \Delta t}
\widehat{M}\big(\boldsymbol{\bar{\rho}}_j^k\big)
\end{aligned}
\end{cases}
\end{align*}
In the sequel, we consider a linear hesitation function and { a relaxation towards the LWR
model, i.e}
$$
\begin{cases}
\begin{aligned}
\partial_t \rho + \partial_x(\rho v)&=0,\\
\partial_t \big(v+\rho\big) + v \partial_x\big(v+\rho\big)
&=
\frac{1}{\tau}
\Big(\Veq(\rho)-v\Big)
\end{aligned}
\end{cases}
\quad
\text{with equilibrium velocity}
\qquad
\Veq(\rho)=1-\rho
$$ 
{and normalized density in the equilibrium model.} 
According to Corollary~\ref{Corollary} the sub-characteristic condition is fulfilled and solutions to the ARZ model are expected to be close to the LWR model if the relaxation parameter $\tau>0$ is sufficiently small. 
Moreover reference solutions are provided, where a Monte-Carlo method is applied to the analytical solution with $M=10^6$~uniformly distributed samples~$\rho_\ell(\xi) \sim \mathcal{U}$ for either of the following Riemann problems:
\begin{align}\tag{shock}
\rho(x,0,\xi)&=\begin{cases}
\rho_\ell(\xi) \sim\mathcal{U}(0.15,0.45)   &\text{for }\ \ x<1,\\
0.7 &\text{for }\ \ x> 1,
\end{cases}
&v(x,0,\xi)&=\begin{cases}
0.7    &\text{for }\ \ x<1,\\
0.3 &\text{for }\ \ x> 1,
\end{cases} \\
\rho(x,0,\xi)&=\begin{cases}
\rho_\ell(\xi) \sim \mathcal{U}(0.55,0.85)     &\text{for }\ \ x<1,\\
0.3 &\text{for }\ \ x> 1,
\end{cases}
&v(x,0,\xi)&=\begin{cases}
0.3    &\text{for }\ \ x<1,\\
0.7 &\text{for }\ \ x> 1. 
\end{cases}
\tag{rarefaction}
\end{align}

{
\subsection{Homogeneous case}\label{sec:num_hype}
This section illustrates the hyperbolic character of the derived stochastic Galerkin formulation, in particular the statement of Theorem~\ref{LEMMA_1}. 
Figure~\ref{plot1} and~\ref{plot2} illustrate the solution to the stochastic Galerkin formulation to the Haar basis with level~$J$. The mean of the density is given by the mode~$\hat{\rho}_0(t,x)$ and plotted as blue line. The confidence region to the truncated gPC expansion is black shaded. Furthermore, the Monte-Carlo confidence region is shown as black dotted line and the reference mean as green dashed line. We observe from Figure~\ref{plot1} for the rarefaction wave that the confidence region is already well captured for level $J=0$ and the mean for~$J=3$. 

Likewise, Figure~\ref{plot2} shows the approximation for the shock case, when each realization admits a discontinuity. The mean, however, is smooth as an average of discontinuous functions. The stochastic Galerkin formulation approximates the mean as step functions (blue line). 
This behaviour is typical and has been observed also for continuous input distributions~\cite{H0,Pettersson2009,FettesPaper,GersterHertyCicip2020}.

\begin{figure}[h!] 
	\scalebox{1}{\includegraphics[width=\linewidth]{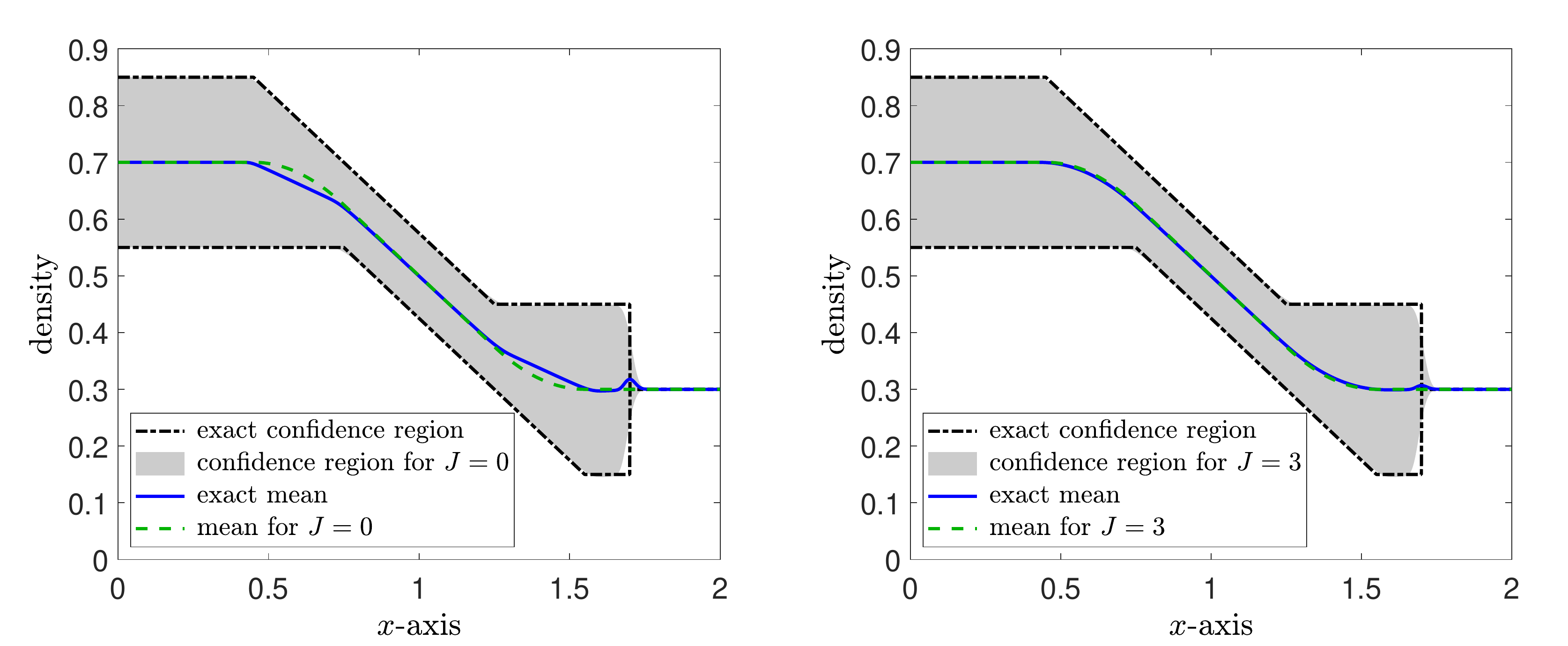}}	
	\caption{Solution to the rarefaction wave at time $t=1$ with discretization~$\Delta x = 0.001$, $\textup{CFL}=0.45$ and Monte-Carlo reference solution with~$M=10^6$ samples. }
	\label{plot1}
\end{figure}

\begin{figure}[h!] 
	\scalebox{1}{\includegraphics[width=\linewidth]{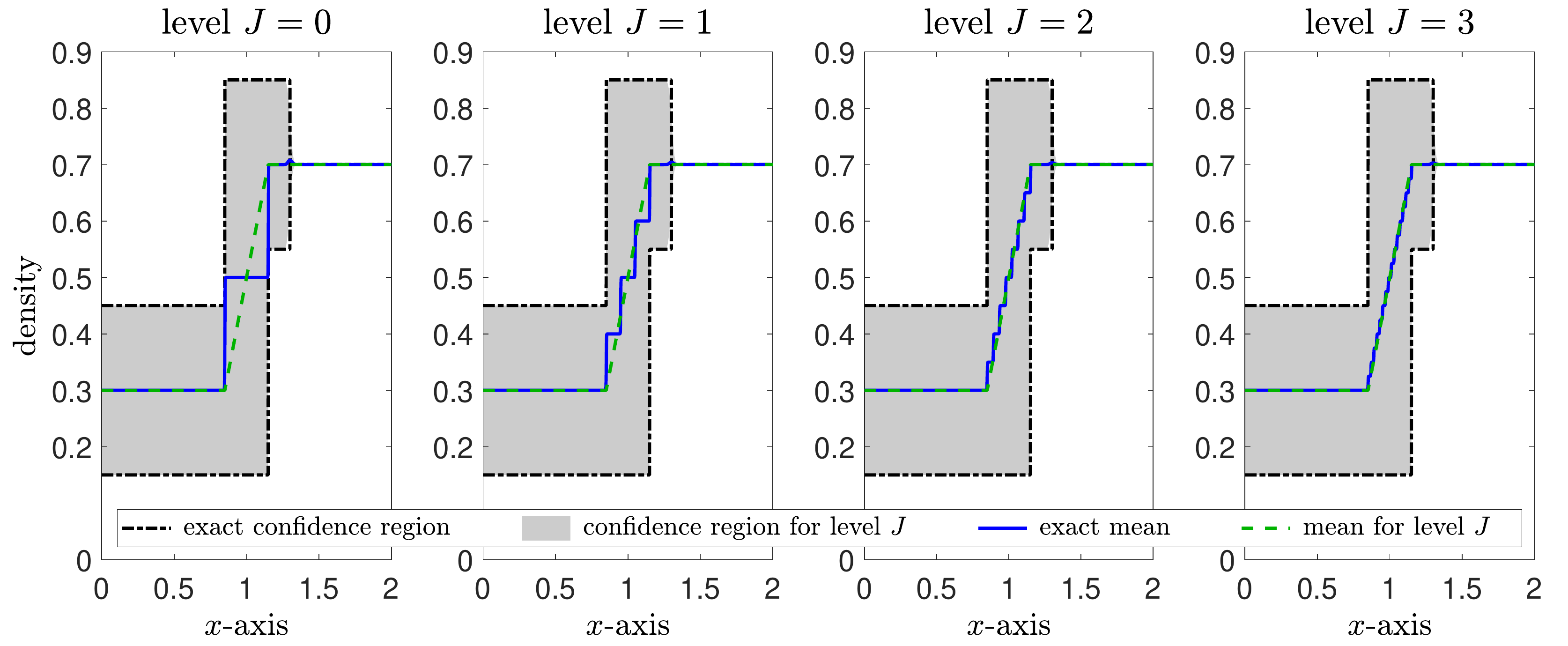}}	
	\caption{Solution to the shock wave at $t=1$ with discretization~$\Delta x = 0.001$, $\textup{CFL}=0.45$ and Monte-Carlo reference solution with~$M=10^6$ samples. }
	\label{plot2}
\end{figure}

\subsection{Inhomogeneous case}

This section is devoted to the stability analysis in~Section~\ref{sec:inhomo}. We investigate the guaranteed dissipativity condition of Theorem~\ref{thm:inhomog} and Corollary~\ref{Corollary}, which presume a relaxation to a first-order model.  
Figure~\ref{plot3} and~\ref{plot4} show  the  behaviour of the inhomogeneous ARZ model for various relaxation parameter, including the limit~$\tau=0$. 
The left panels show the results for the level $J=2$ without relaxation and the exact confidence regions are plotted in the remaining panels for comparison. Indeed, we observe a convergence towards the LWR model according to Corollary~\ref{Corollary}.

\begin{figure}[h!] 
	\scalebox{1}{\includegraphics[width=\linewidth]{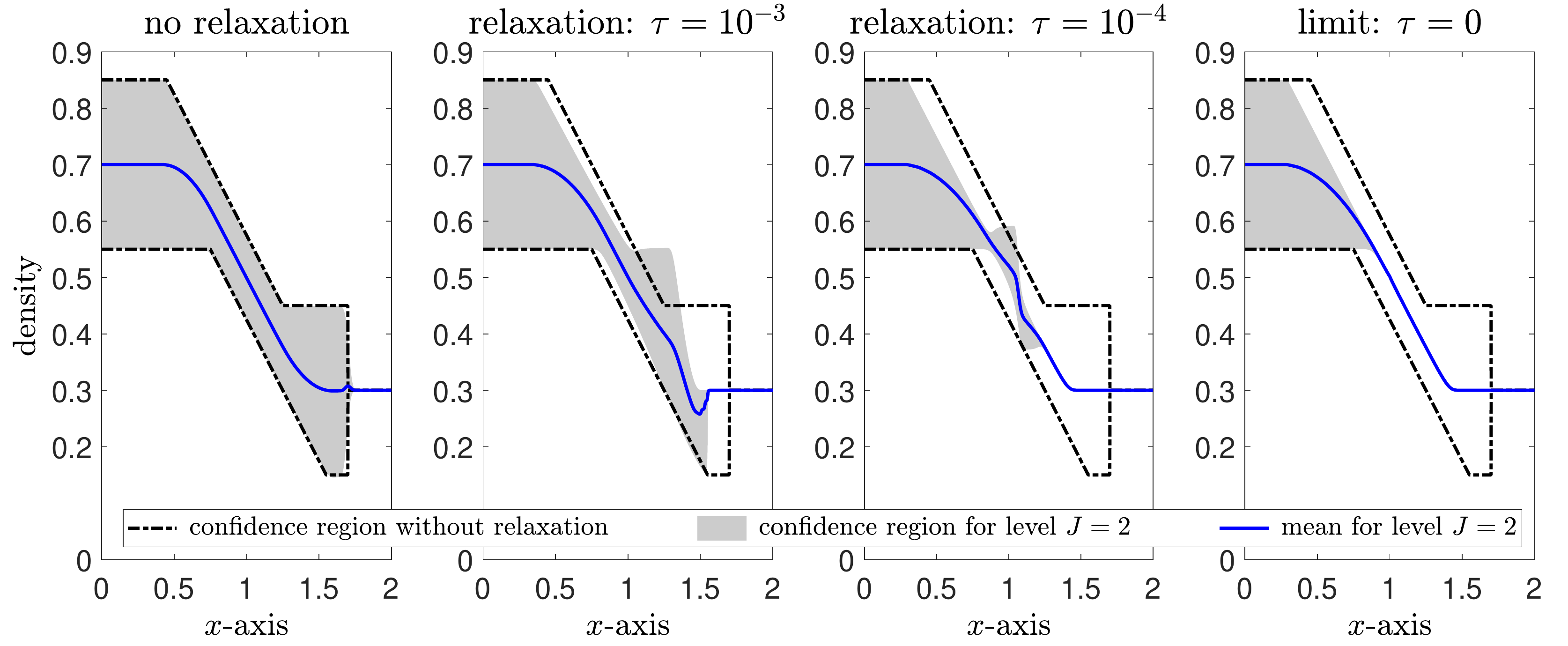}}	
	\caption{Solution to the inhomogeneous ARZ model for the rarefaction case.}
	\label{plot3}
\end{figure}

\begin{figure}[h!] 
	\scalebox{1}{\includegraphics[width=\linewidth]{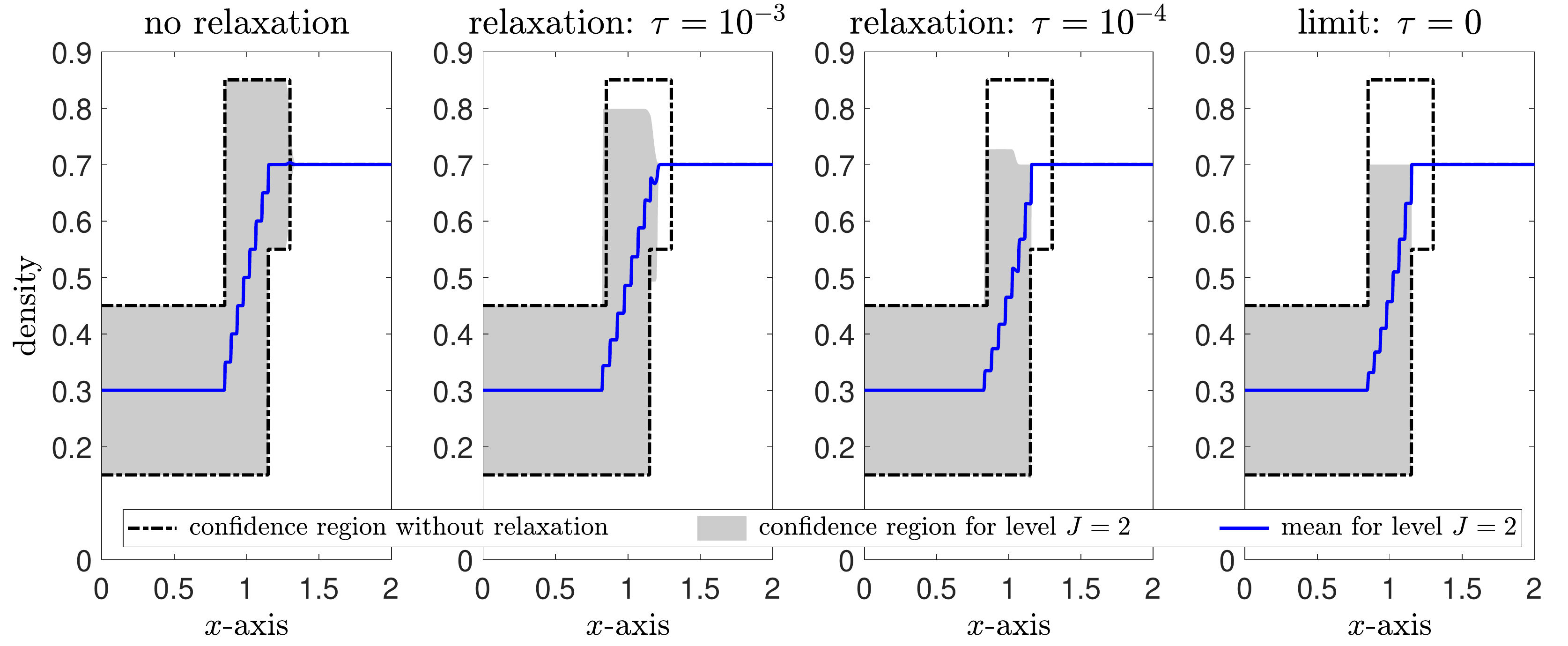}}	
	\caption{Solution to the inhomogeneous ARZ model for the shock case.}
	\label{plot4}
\end{figure}

}
\section{Summary}\label{sec:conclusion}

A stochastic Galerkin formulation of the Aw-Rascle-Zhang (ARZ) model has been presented. In particular, hyperbolicity has been shown for a special class of wavelet-based expansions. The analysis is based on a \emph{non-conservative formulation}. This  allows a stability analysis for the inhomogeneous ARZ with stiff relaxation, when solutions are expected to be close to an equilibrium velocity that satisfies a scalar conservation law. 
{ Due to the non-conservative formulation, the derived theoretical results hold \emph{only for smooth solutions}. 
However, a relationship to a conservative form has been established.
This allows for a numerical discretization with an IMEX scheme.}


\section*{Acknowledgments}
\noindent
This research is funded by the Deutsche Forschungsgemeinschaft (DFG,
German Research Foundation) under Germany's Excellence Strategy –
EXC-2023 Internet of Production – 390621612 and by
DFG~HE5386/18,19, DFG~320021702/GRK2326. \vspace{2mm}

\noindent
Furthermore, we would like to offer special thanks to Giuseppe Visconti.

\bibliography{mybib}
\bibliographystyle{abbrv}

\end{document}